\documentclass[11pt]{article}
\usepackage{amsfonts,amsmath,amsthm}
\usepackage{graphicx}

\usepackage[usenames,dvipsnames]{color}
\usepackage[colorlinks,%
linkcolor=BrickRed,%
filecolor =BrickRed,%
citecolor=RoyalPurple,%
]{hyperref}

\def\bl#1{{\color{blue} #1}}

\def\urls#1{{\small \url{#1}}}

\def\transpose{{\hbox{\tiny\it T}}}

\newcommand{\RL}{{\mathbb R}}

\newcommand{\clH}{\mbox{${\cal H}$}}

\newcommand{\clM}{\mbox{${\cal M}$}}

\newcommand{\clS}{\mbox{${\cal S}$}}

\def\be{\begin{eqnarray}}
\def\ee{\end{eqnarray}}
\def\ben{\begin{eqnarray*}}
\def\een{\end{eqnarray*}}

\def\barc{\oo c}
\def\barx{{\oo x}}

\def\sq{$\Box$}

\def\qed{\ifmmode\sq\else{\unskip\nobreak\hfil
\penalty50\hskip1em\null\nobreak\hfil\sq
\parfillskip=0pt\finalhyphendemerits=0\endgraf}\fi\par\medbreak}

\def\tov{\mathchoice{{\ \mathop{\displaystyle\mathop{\longrightarrow}_{v}}\ }}%
{{\,\mathop{\displaystyle\mathop{\longrightarrow}_{v}}\,}}%
{{\,\mathop{\displaystyle\mathop{\to}_{v}}\,}}%
{{\,\mathop{\displaystyle\mathop{\to}_{v}}\,}}%
}

\def\tove{\mathchoice{{\ \mathop{\displaystyle\mathop{\longrightarrow}_{v^{\eta}}}\ }}%
	{{\,\mathop{\displaystyle\mathop{\longrightarrow}_{v^{\eta}}}\,}}%
	{{\,\mathop{\displaystyle\mathop{\to}_{v^{\eta}}}\,}}%
	{{\,\mathop{\displaystyle\mathop{\to}_{v^{\eta}}}\,}}%
}

\newsavebox{\junk}
\savebox{\junk}[1.6mm]{\hbox{$|\!|\!|$}}
\def\lll{{\usebox{\junk}}}

\def\limsup{\mathop{\rm lim\ sup}}

\def\state{{\sf X}}

\newcommand{\field}[1]{\mathbb{#1}}

\def\Re{\field{R}}

\def\nat{\field{Z}_+}

\def\Co{\field{C}}

\def\ind{\field{I}}

\def\One{\hbox{\large\bf 1}}

\def\bfmN{{\mbox{\protect\boldmath$N$}}}
\def\bfmS{{\mbox{\protect\boldmath$S$}}}

\def\haP{{\widehat P}}

\def\til={{\widetilde =}}

\def\tilP{{\widetilde P}}

\def\clB{{\cal B}}

\def\clD{{\cal D}}

\def\clH{{\cal H}}

\def\clM{{\cal M}}

\def\clP{{\cal P}}

\def\eqdef{\mathbin{:=}}

\def\Prob{{\sf P}}

\def\Expect{{\sf E}}

\def\epsy{\varepsilon}

\def\varble{\,\cdot\,}

\newtheorem{theorem}{Theorem}[section]

\newtheorem{proposition}[theorem]{Proposition}
\newtheorem{lemma}[theorem]{Lemma}

\def\Lemma#1{Lemma~\ref{#1}}
\def\Proposition#1{Proposition~\ref{#1}}
\def\Theorem#1{Theorem~\ref{#1}}

\def\Section#1{Section~\ref{#1}}

\def\bdes{\begin{description}}
\def\edes{\end{description}}

\newcommand{\oo}{\overline}

\def\Lv#1{L_\infty^{v,#1}}
\def\Lve#1{L_\infty^{v^\eta,#1}}
\def\Lvx{L_\infty^{v}}
\def\Lvex{L_\infty^{v^\eta}}

\def\bfmX{{\mbox{\protect\boldmath$X$}}}

\def\tilc{\tilde c}

\def\transpose{{\hbox{\tiny\it T}}}

\def\One{\mbox{\rm{\large{1}}}}

\newcounter{rmnum}
\newenvironment{romannum}{\begin{list}{{\upshape (\roman{rmnum})}}{\usecounter{rmnum}
\setlength{\leftmargin}{24pt}
\setlength{\rightmargin}{16pt}
\setlength{\itemindent}{-1pt}
}}{\end{list}}
 
\newcounter{anum}
\newenvironment{alphanum}{\begin{list}{{\upshape (\alph{anum})}}{\usecounter{anum}
\setlength{\leftmargin}{24pt}
\setlength{\rightmargin}{16pt}
\setlength{\itemindent}{-1pt}
}}{\end{list}}

\newlength{\noteWidth}
\long\def\notes#1{\ifinner
             {\tiny #1}
             \else
             \marginpar{\parbox[t]{\noteWidth}{\raggedright\tiny #1}}
             \fi}

\def\archive#1{}

\newcommand{\Chi}{{%
\mathchoice{\mathord{\raisebox{1.5pt}{\scalebox{1.25}{$\chi$}}}}%
{\mathord{\raisebox{1.5pt}{\scalebox{1.25}{$\chi$}}}}%
{\mathord{\raisebox{1pt}{\scalebox{.75}{$\chi$}}}}%
{\mathord{\raisebox{.8pt}{\scalebox{.6}{$\chi$}}}}%
}}

\def\Sens{{\cal S}}
\def\dffA{{\cal A}}

\def\tilgrad{\nabla^\Sens\!}

\begin{document}
	
 
\title{\vspace{-2cm}%
Geometric Ergodicity in a Weighted Sobolev Space}

\author{Adithya Devraj\thanks{{\bf Corresponding author}.  Department of Electrical and Computer 
                Engineering,
                University of Florida, Gainesville, USA.
                 Email: {\tt adithyamdevraj@ufl.edu}}
\and
      	Ioannis Kontoyiannis\thanks{
		Department of Informatics,
		Athens University of Economics and Business,
		Patission 76, Athens 10434, Greece.
                Email: {\tt yiannis@aueb.gr}.        
		I.K.\ was supported 
	by the European Union and Greek National 
	Funds through the Operational Program Education 
	and Lifelong Learning of the National Strategic Reference 
	Framework through the Research Funding Program Thales-Investing 
	in Knowledge Society through the European Social Fund.  
	}
\and
        Sean Meyn\thanks{Department of Electrical and Computer 
                Engineering,
                University of Florida, Gainesville, USA.
                 Email: {\tt meyn@ece.ufl.edu}.
              S.M.\ and A.D.\ were supported 
              by NSF grants EPCN-1609131 and CPS-1259040.
                }
            }

\maketitle
\thispagestyle{empty}
\setcounter{page}{0}
 
\begin{abstract}
\noindent
For a discrete-time Markov chain $\bfmX=\{X(t)\}$ evolving 
on $\Re^\ell$ with transition kernel $P$,
natural, general conditions are developed under which
the following are established:
\begin{itemize}
\item[(i)]
The transition kernel $P$ has a purely discrete spectrum, 
when viewed as a linear operator on a weighted Sobolev 
space $\Lv{1}$ of functions  with norm,
\[
\|f\|_{v,1} = \sup_{x\in \Re^\ell} \frac{1}{v(x)} \max 
\{|f(x) |, |\partial_1 f(x)|,\ldots,|\partial_\ell f(x)|\},
\]
where $v\colon \Re^\ell \to [1,\infty)$ is a Lyapunov function
and $\partial_i:=\partial/\partial x_i$.

\item[(ii)]
The Markov chain is geometrically ergodic in $\Lv{1}$:
There is a unique invariant probability measure $\pi$ and constants 
$B<\infty$ and $\delta>0$ such that,
for each $f\in \Lv{1}$, any initial condition $X(0)=x$,
and all $t\geq 0$:
$$\Big|\Expect_x[f(X(t))] - \pi(f)\Big|
\le Be^{-\delta t}v(x),\quad
\|\nabla\Expect_x[f(X(t))] \|_2 
\le Be^{-\delta t} v(x),
$$
where $\pi(f)=\int fd\pi$.
\item[(iii)] 
For any function $f\in \Lv{1}$ there is a function $h\in \Lv{1}$ solving 
Poisson's equation:
\[
h-Ph = f-\pi(f).
\]
\end{itemize}
Part of the analysis is based on an operator-theoretic 
treatment of the sensitivity process that appears 
in the theory of Lyapunov exponents.   

\bigskip

\bigskip

{\small
\noindent
\textbf{Keywords:}  
Markov chain,  
stochastic Lyapunov function,
discrete spectrum,
sensitivity process,
weighted Sobolev space,
Lyapunov exponent}

\bigskip

{\small
\noindent
\textbf{2010 AMS Subject Classification:}
60J05,          
60J35,		
37A30,          
47H20.          
}

\end{abstract}

\thispagestyle{empty}
%
%
%
%
%
%
%
%

\clearpage

%

\section{Introduction}

Consider a discrete-time Markov chain $\bfmX = \{X(t) : t\geq 0\}$ taking
values in   $\state=\Re^\ell $, equipped with its associated Borel 
$\sigma$-field $\clB$.   Throughout the paper
(except where explicitly noted otherwise, in particular see
\Section{s:wo}) the process $\bfmX$
is assumed to be defined by the nonlinear state space model,  
\begin{equation}
X(t+1) = a(X(t),N(t+1)), \quad t \in \nat,
\label{e:SP_disc}
\end{equation}
where $\bfmN = \{N(t): t=0,1,2,\ldots\}$ is a sequence
of $\Re^m$-valued, independent and identically distributed
random variables, and $a: \Re^{\ell\times m} \to \Re^\ell$ 
is continuous, so that each realization $X(t)$ is a continuous 
function of $X(0)=x$.

The distribution of $\bfmX$ is 
described by its initial state $X(0)=x\in\state$
and its transition semigroup:
For any $t\ge 0$, $x\in\state$,  $A\in \clB$,
$$
P^t(x,A):=\Prob_x\{X(t)\in A\}:=\Pr\{X(t)\in A\,|\,X(0)=x\},
$$
with the usual convention that $P^1$ is simply denoted $P$.
For the Markov chain described by~\eqref{e:SP_disc},  it follows that $
P(x,A) = \Pr\{a(x,N(1)) \in A\}$.

Recall that the kernel $P^t$
acts as a linear operator on functions $f:\state\to\RL$ on the
right and on signed measures $\nu$ on $(\state,\clB)$ on
the left, respectively, as, 
$$
P^tf\,(x)=\int f(y)P^t(x,dy),
\;\;\;\;
\nu P^t\, (A) = \int \nu(dx)P^t(x,A),
\;\;x\in\state,\,A\in\clB,
$$
whenever the integrals exist.  Also, for any
signed measure $\nu$ on $(\state,\clB)$ and any
function $f:\state\to\RL$ we write $\nu(f):=\int f d\nu$,
whenever the integral exists.  In this paper 
we constrain the domain of functions $f$ to a Banach space 
defined with respect to a weighted $L_\infty$ norm.  

Specifically, given a fixed continuous function $v\colon\state\to [1,\infty)$,
the $v$-norm of any measurable function $f\colon\state\to\Re$
is denoted,
\begin{equation}
\|f\|_v \eqdef \sup_x  \frac{|f(x)|}{v(x)},
\label{e:Lvx}
\end{equation}
and the corresponding
Banach space $\Lvx$ is defined as,
$\Lvx  \eqdef \{ f\colon\state\to\Re : \|f\|_{v} <\infty\}$. 
An analogous weighted norm is defined for
signed measures $\mu$ on $(\state,\clB)$
via,
\[
\|\mu\|_v
	:=\sup\Big\{\frac{|\mu(h)|}{\|h\|_v}  :  
	h\in \Lvx,\ \|h\|_{v}\neq 0\Bigr\},
\]
and we denote by $\clM_1^v$ 
the space of signed measures $\mu$ with $\|\mu\|_v<\infty$.

The Markov chain $\bfmX$ is
\textit{$v$-uniformly ergodic} \cite{MT,konmey03a}
whenever there exists a function $v$,  
a unique invariant probability measure 
$\pi$, and constants $b_0 < \infty$ and $0<\rho_0 < 1$,  such that,
for each function $f\in\Lvx$,
\begin{equation}
\big| \Expect[f(X(t))\mid X(0)=x] - \pi(f) \big| \leq b_0 \rho_0^t  
\|f\|_v v(x) ,\quad t\ge 0\,,
\label{e:v-uni}
\end{equation}
where $\pi(f)=\int f d\pi$.
It is well known that this is equivalent to the existence 
of a Lyapunov function satisfying the drift condition~(V4) of \cite{CTCN}. 

\subsection{Motivation and background}

\archive{\bl{ad: $\alpha$ is not defined.}}

Let $c:\state\to\Re$ be a given function on the state space
of $\bfmX$. One starting point for the
classical study of the long-term behaviour of 
$\bfmX$ is the development of conditions for
the existence of 
the mean ergodic limit,
\begin{equation}
\barc := \lim_{n\rightarrow\infty} \frac{1}{n} 
\sum_{t=0}^{n-1} \Expect [c(X(t)) \mid X(0) = x]\, ,
\label{e:barc}
\end{equation}
and of the function,
\begin{equation}
h(x):=  \sum_{t=0}^{\infty} \Expect [c(X(t)) -\barc \mid X(0) = x]
\, , \quad x\in\state,
\label{e:fish-sum}
\end{equation}
which can be shown to be a solution of the associated
{\em Poisson equation},
\begin{equation}
h(x) - \Expect [h(X(1)) \mid X(0) = x]
= c(x)-\barc
\, , \quad x\in\state.
\label{eq:Poisson}
\end{equation}
For example, if $\bfmX$ is $v$-uniformly ergodic,
then in addition to the convergence~(\ref{e:v-uni}) of $P^t(x,\cdot)$ 
to its unique invariant probability measure $\pi$,
the ergodic averages of $c(X(t))$ converge a.s.\
to $\barc=\pi(c)$, and their associated central-limit-theorem
variance is naturally expressed in terms of $h$ \cite{MT,asmgly07},
$$\sigma^2=\pi\left(h^2-(Ph)^2\right).$$
Moreover, if $c\in\Lvx$ then $h$ is also in $\Lvx$ 
\cite{glymey96a}.

A closely related object of 
interest is the collection, for each $\alpha\in(0,1)$, 
of the functions,
\begin{equation}
h_\alpha(x) :=  \sum_{t=0}^{\infty} \alpha^t 
\Expect [c(X(t)) \mid X(0) = x]
\, , \quad x\in\state,
\label{e:DCOE_disc}
\end{equation}
where each $h_\alpha$ can be viewed as the result of the action
of the resolvent kernel, 
$$R_\alpha:=\sum_{t=0}^\infty\alpha^tP^t,$$ 
on the function $c$. Again, under $v$-uniform ergodicity,
$h_\alpha\in\Lvx$ for any $\alpha<1$,
whenever $c\in\Lvx$ \cite{konmey03a}.

The main goal of the present work is to develop natural conditions
that guarantee appropriate {\em smoothness} properties 
of $\barc$, $h$ and $h_\alpha$. In particular, as described next,
we show that the derivative of $P^tc(x)=\Expect[c(X(t))|X(0)=x]$ 
with respect to the
initial condition $X(0)=x$ converges to zero; we provide series 
representations, analogous to (\ref{e:fish-sum}) 
and (\ref{e:DCOE_disc}), for the {\em derivatives} of $h$ and 
$h_\alpha$; and we also obtain bounds 
for those derivatives.

In addition to the theoretical interest 
of these results, we are also motivated in part by related
questions and applications in stochastic control.
In that context, 
$c$ is viewed as a {\em one-step cost function},
$\alpha$ is the {\em discount factor}, 
$\barc$ is the {\em average cost}, 
$h(x)$ is the \emph{relative value function},
and $h_\alpha(x)$ is the {\em total discounted cost}.
The present results provide a foundation for 
a new approach to approximate dynamic programming developed 
in \cite{devmey16a}, and to gain approximation for the feedback 
particle filter \cite{laumehmeyrag15,raddevmey16}.

%

\subsection{Overview of main results}

Suppose that the function  $a$ appearing in \eqref{e:SP_disc}
is continuously differentiable.   This justifies the following definition of 
the $\ell\times \ell$ \textit{sensitivity process} 
$\bfmS=\{\Sens(t):t\geq 0\}$,  whose $(i,j)$ component 
is defined at time $t$ by:
\begin{equation}
\Sens_{i,j} (t) \eqdef \frac{\partial X_i(t)}{\partial X_j(0)}, \quad 1\leq i,j \le \ell.
\label{e:SensDef}
\end{equation}
From \eqref{e:SP_disc}, the sensitivity process evolves
according to the random linear system,
\begin{equation}
\begin{aligned}
\Sens(t+1) &= \dffA(t+1) \Sens(t) ,\quad \Sens(0)=I\,, 
\end{aligned}
\label{e:SP_sens_disc}
\end{equation}
where $\dffA^\transpose(t) \eqdef \nabla_x a\, (X(t-1),N(t))$.

For any function $f\in C^1$ and all $t\ge 0$,
we write:
\begin{equation}
\tilgrad f \, (X(t)) \eqdef \Sens^\transpose(t) \nabla f(X(t))\, .
\label{e:tilgrad}
\end{equation}
It follows from the chain rule that this coincides with the gradient of 
$f(X(t))$ with respect to the initial condition $X(0)$.  This interpretation 
of \eqref{e:tilgrad} motivates the introduction of a new semigroup 
$\{Q^t: t\ge 0\}$ of operators acting on measurable  functions 
$g: \state \to \Re^\ell$:
For $t\ge 1$, 
\begin{equation} 
Q^t g(x) \eqdef \Expect \bigl[ \Sens^\transpose(t) 
g(X(t))\mid X(0)=x  \bigr]\, ,
\label{e:Qt} 
\end{equation}
and  $Q^0g =g$.
Provided we can exchange the gradient and the expectation,  
and writing as usual $\Expect_x(\cdot)$ for the
conditional expectation $\Expect(\cdot|X(0)=x)$,
\begin{equation*}
\frac{\partial}{\partial x_i} \Expect_x[f(X(t)) ]  
= \Expect_x \bigl[ [\tilgrad f(X(t))]_i  \bigr],
\label{e:tilgradMean}
\end{equation*}
which implies that:
$$  \nabla P^t f(x)  = \Expect_x [\tilgrad f(X(t))]  = Q^t \nabla f\,(x).$$

\medskip

\noindent
{\bf Main results. }
The main contribution of this paper is the justification
of the above manipulations, within an appropriate Banach space 
setting. Specifically, for all functions                                    
$c:\state\to\RL$ in an appropriate space, 
we identify general, natural conditions under 
which the following are established:
\begin{romannum}
\item  
Not only does $P^t c$ converge to $\barc=\pi(c)$ as $t\to\infty$, 
but also the gradient $\nabla P^t c$ of $P^tc$ with respect to 
the initial condition $X(0)=x$ converges to zero, at a uniformly 
geometric rate;
cf.~Theorem~\ref{t:ergoMain}.

\item
The solution $h$ of the Poisson equation defined in~(\ref{e:fish-sum})
is differentiable, and the following
representation is obtained in Theorem~\ref{t:fish}
for its gradient,
$$\nabla  h = \Omega\nabla  c \eqdef \sum_{t=0}^{\infty} 
Q^t\nabla  c,$$
where $\{Q^t\}$ is the semigroup defined
in~(\ref{e:Qt}) in terms of the sensitivity process.

\item 
Similarly, for any $\alpha\in(0,1)$, the following 
representation is obtained in Theorem~\ref{t:discfish} 
for the gradient of the total discounted cost $h_\alpha$
defined in~(\ref{e:DCOE_disc}):
\begin{equation}
\nabla  h_\alpha = \Omega_\alpha \nabla  c \eqdef \sum_{t=0}^{\infty} \alpha^t Q^t  \nabla  c .
\label{e:OmDream}
\end{equation} 
\end{romannum}

\subsection{Prior research}

The sensitivity process defined in \eqref{e:SensDef} is used to 
define the Lyapunov exponent,
\[
\Lambda_1 := \lim_{t\to\infty} \frac{1}{t} \log ( \|  \Sens(t) \| );
\]
here and in (\ref{eq:Lambdap},\ref{e:Assumption4:haimat08}),
$\|\cdot\|$ can be taken to be
any matrix norm.
A negative exponent implies a topological notion of coupling: 
Suppose that $\Lambda_1$ is a negative constant, independent of 
the initial condition.  If $\bfmX$ and $\bfmX'$ are two realizations 
of the Markov chain with 
different initial states, it follows from the mean value theorem that,
with probability one,
\[
\lim_{t\to\infty} \| X(t) - X'(t) \|_2 = 0,
\]
and that this convergence is geometrically fast,
with rate $e^{t\Lambda_1}$. 
Much of the earlier relevant research, including the study
of the corresponding $p$th mean,
\begin{equation}
\bar{\Lambda}_p 
:= \lim_{t\to\infty} \frac{1}{t} \log \big(\Expect_x
\big[ \|  \Sens(t) \|^p\big]  \big),
\label{eq:Lambdap}
\end{equation}
is for diffusion processes in continuous time \cite{kun90,bax89,arncra91}. 

Verifiable conditions for a negative Lyapunov exponent are established in
\cite{atazei97} for a class of hidden Markov models, and in \cite{huakonmey02a} for a general class of stochastic sequences of the form~\eqref{e:SP_disc};
coupling results that suggest a negative Lyapunov exponent  
are established in \cite{haimat08} for a class of diffusions.    
In these papers the main results are established 
\textit{without} $\psi$-irreducibility.   As discussed 
in \cite[Section~6.4]{MT},  in such cases it is impossible to establish 
convergence of  the Markov semigroup in total variation, so it is natural 
to instead rely on topological notions of convergence or coupling.  

The prior work most closely related to our results is \cite{haimat08},
although the development there is entirely for continuous-time
processes. A central goal of \cite{haimat08} is to obtain coupling 
bounds in a topological sense for a class of stochastic Navier-Stokes 
equations, but results for a general stochastic flow on a 
Banach space are also derived. Assumption~4 of \cite{haimat08} 
imposes a contraction bound directly on the sensitivity process,
which, in the finite-dimensional case
implies that, for a function $v\ge 1$,   
constants $k<\infty$, $\rho<1$, and some time $t_0>0$,    
\begin{equation}
\Expect_x [v(X(t_0)) ]  +\Expect_x [  \|  \Sens (t_0)   \|   v(X(t_0)) ]    \le k v(x)^\rho  \,,\qquad x\in\state\, .
\label{e:Assumption4:haimat08}
\end{equation}
This and other assumptions imply the desired coupling result in 
\cite[Theorem~3.4]{haimat08}, which is also shown to imply the
ergodic limit in our Theorem~\ref{t:ergoMain}.
And we should note that the weighted Sobolev norm $\|f\|_{v,1}$ used 
in Theorem~\ref{t:ergoMain} and throughout in this paper
(cf.~(\ref{e:fv1}) in Section~\ref{s:MainResults}
below), also appears as $\|f\|_{V^r}$  in \cite[p.~21]{haimat08}.

The present approach is complementary to \cite{haimat08}.  
Their conclusions are far stronger than those presented here:  
They obtain a very strong topological coupling of the process, 
and they do not require $\psi$-irreducibility.   However, 
these strong conclusions require strong assumptions.  Most significant is 
that their assumptions imply the contraction 
bound \eqref{e:Assumption4:haimat08}
that is not easily verified in applications,  and is unlike any assumption 
imposed in the present work. 

Although the key results of this paper are related in spirit 
to much of the prior work mentioned above, there are 
no formal implications, in either direction, 
to existing results that we are aware of.  
In particular, it is not known whether the assumptions imposed here 
can be used to establish any 
form of topological coupling. And, rather than the norm of 
the sensitivity process as in the definition of $\bar{\Lambda}_1$,
we obtain bounds on the 
expectation of the sensitivity process,
showing, for example, that 
for all $C^1$ functions $g\colon\state\to\Re$
in an appropriate Banach space,
the following can be uniformly bounded above: 
\[
  \limsup_{t\to\infty} \frac{1}{t} \log \bigl| \Expect_x\big\{
	[\Sens(t) \nabla g \, (X(t))]_i \big\}  \bigr|, \quad 1\leq i \le \ell.
\] 
The relationship between the limit theorems established in this paper and 
classical Lyapunov exponents is a topic of current research.
  
Our main assumptions are minor variants of those used in 
much of the prior work on ergodic theory for Markov chains.   In particular,  
the Lyapunov drift condition (DV3) is assumed throughout:  For nonnegative,
continuous functions $V\colon\state\to\Re_+$, $W\colon\state\to [1,\infty)$,   
$\delta>0$, and a compact set $C$:
\begin{equation}
\log \Expect\bigl[\exp\{V(X(t+1)) - V(X(t)\} \mid X(t) 
= x\bigr] \leq -\delta W(x)\,,\quad x\in C^c.
\label{e:DV3}
\end{equation} 
Condition (DV3) is
an essential ingredient 
in much of the prior work of Donsker and Varadhan   on large deviation theory 
for Markov  models  \cite{donvarI-II,donvarIII,donvarIV},
and it is used to bound rates of convergence for a Markov chain 
for both mean and ``multiplicative'' ergodic theory 
in \cite{balmey00a,konmey05a,konmey17a}.     
Also, the discrete-time counterpart of 
\cite[Assumption~4, equation~(15)]{haimat08} implies
\eqref{e:Assumption4:haimat08} with $V$ having bounded sublevel sets,
and hence (DV3) for $W=V=\log v$.


The value of (DV3) is most clear when the sublevel sets of the function $W$ 
are compact.   In this case,  an $n$-step transition kernel can be 
approximated by its truncation to a compact set arbitrarily closely in an 
associated induced operator norm   \cite{balmey00a,konmey05a,konmey17a};
see also \cite{wu95a,wu00b} on the implications 
of truncation approximations.  
The main assumptions of the paper summarized in \Section{s:main} 
impose (DV3) and minor additional assumptions so that a truncation 
approximation is valid in the stronger norm used in this paper.

There has been increasing interest in finding connections between (DV3) and
logarithmic {Sobolev} or {Poincar\'e} inequalities 
\cite{gonwu06a,catguiwanwu08,catgui17}.  
The implications of this and similar drift conditions are the main focus 
of \cite{MT}.  In particular, in this monograph and subsequent papers 
\cite{glymey96a,konmey03a,konmey12a}, drift conditions are used  to obtain 
existence and bounds on solutions to Poisson's equation.  A log-Sobolev 
inequality   is the condition used in \cite{laumehmeyrag15} to establish the 
existence of a smooth solution to Poisson's equation for a diffusion.  
Somewhat more explicit sufficient conditions for a smooth solution 
to Poisson's equation are obtained in \cite{parver01} for elliptic diffusions.


Poisson's equation is one tool used in addressing parametric sensitivity 
in Markov chains, starting with the 50-year-old work of Schweitzer 
\cite{sch68};    \textit{infinitesimal perturbation analysis} is a well-known 
application of these techniques \cite{caoche97}.   A modern treatment is 
contained in the very recent work  \cite{rhegly17}.    The focus of this 
paper is on sensitivity with respect to the initial condition of the Markov 
chain rather than parametric uncertainty,  so there is no obvious 
relationship with this prior research.   

\medskip

The remainder of the paper is organized as follows: 
\Section{s:MainResults} summarizes the main results for the Markov 
chain \eqref{e:SP_disc};  these results are obtained under a Lyapunov 
condition slightly stronger than what is assumed in \cite{konmey05a}.   
\Section{s:spectral} contains proofs of the main results, leaving technical 
results to the Appendix.  \Section{s:wo} contains results for a 
general Markov chain,
not necessarily admitting the representation \eqref{e:SP_disc}.     

 
\section{Assumptions and Main Results}
\label{s:MainResults}

The four assumptions (A1)--(A4) introduced in this section include 
the existence of 
a Lyapunov function $V:\state\to(0,\infty)$ satisfying the drift 
condition (DV3) of \cite{konmey03a,konmey05a}; see condition~(A4) below.    
We denote $v= e^V$,  
which is used to define the norm $\|\varble\|_v$ in \eqref{e:Lvx}.

The weighted Sobolev spaces $\Lv{k}$
considered in this paper are based on a function-space norm 
that involves the derivatives of a function $f\colon\state\to\Re$.   
For each $k\ge 1$ denote,
\begin{equation} 
\| f\|_{v,k} = \max_{ |\alpha|\le k} \| D^\alpha f \|_v,
\label{e:fvk}
\end{equation}
where the maximum is over all multi-indices
$\alpha\in\nat^\ell$ with $\sum_i\alpha_i\leq k$,
and $D^\alpha$ is the 
corresponding partial derivative.  For $k=1$ this is a maximum over 
$\ell+1$ terms,
\begin{equation}
\| f\|_{v,1} = \max \Bigl\{ \|f\|_v,  \| \partial_1 f\|_v,\dots,    
\| \partial_\ell   f\|_v \Bigr\},
\label{e:fv1}
\end{equation}
where $\partial_i$ denotes the first partial derivative
with respect to $x_i$, $\partial/\partial x_i$.
For $k=0$ we let
$\Lv{0}$ denote the space $\{ f\in\Lvx: \text{$f$ is continuous} \}$
with norm $\|\varble\|_v$.  For each $k\ge 1$ we also
define the spaces,
\[
\Lv{k}  \eqdef \{ f\colon\state\to\Re : D^\alpha f \in\Lv{0} \ \text{for all} \  |\alpha|\le k \},
\]
equipped with the norm defined in \eqref{e:fvk}.
This introduces two new restrictions on any function $f\in \Lv{k} $:  
The $k$th partial derivatives of $f$ must exist and be absolutely bounded 
by a constant times $v$.  In addition, $f$ and these derivatives must be 
continuous.   In the special case $v\equiv 1$,  the space $\Lv{k}$ 
coincides with the usual Sobolev space $W^{k,1}$.
Throughout most of the paper we restrict 
attention to the cases $k=0$ and $k=1$.
In Proposition~\ref{t:Banach} we show that 
the normed spaces $\Lv{0}$ and $\Lv{1}$ are complete
and therefore are Banach spaces.

Consideration of the space $\Lv{1}$ requires the following 
assumptions on the evolution equations \eqref{e:SP_disc}. 
Assumption~(A1) 
ensures that the state at each time $t$ is  a continuously differentiable 
function of its initial condition $X(0)=x$, and justifies the representation 
of $\tilgrad f (X(t)) $ in~\eqref{e:tilgrad}. 
$$
\left. 
\mbox{\parbox{.85\hsize}{\raggedright
(i) The  process $\bfmN$ does not depend upon the initial condition $X(0)$.
 
(ii) The function $a$ is continuously differentiable in its first variable, 
with:
\[
\sup_{x,n} \| \nabla a (x,n)\| <\infty.
\]
}}
\right\}
\eqno{\hbox{\bf (A1)}}
$$
The notation $\| \varble \|$ in (ii) can represent any matrix norm, 
and the $j$th column of the $\ell\times\ell$ matrix  $\nabla a$ 
is equal to the gradient of $  a_j$, so that,
\[
[\nabla a (x,n)]_{i,j} \eqdef \frac{\partial }{\partial x_i} a_j (x,n) , \quad 1\leq i,j \leq \ell .
\]

\subsection{Irreducibility, densities and drift}
\label{s:irred}

The general ergodic theory of Markov chains as developed in \cite{MT} involves
two assumptions.  The first is a generalization of irreducibility as defined 
for finite-state space Markov chains, and the second is a 
Foster-Lyapunov drift condition. The irreducibility conditions will hold 
under assumptions~(A2) and~(A3);  the first is a density condition,
and the second is a ``reachability'' assumption:   
$$
\left. 
\mbox{\parbox{.85\hsize}{\raggedright
For some $t_0\ge 1$, the transition kernel admits  a smooth density.  
That is, there is a continuously differentiable function $p_{t_0}$ on $ \state\times \state$ such that,
\[
P^{t_0}(x,A) = \int_A p_{t_0}(x,y)\, dy\,,\qquad x\in\state,\ A\in\clB
\]
 }}
\quad\right\}
\eqno{\hbox{\bf (A2)}}
$$
%
%
\noindent
Under (A2), there is in fact a density for {\em every} $t\ge t_0$, given by:
%
\[
p_t(x,y) = 
\int P^{t-t_0}(x,dz) p_{t_0}(z,y), \qquad t\ge t_0,\quad x,y\in\state.
\]
The representation \eqref{e:SP_disc} implies the Feller property, i.e.,
that the  function $P^tf$ is continuous whenever $f$ is continuous and 
bounded. Assumption~(A2) implies the strong Feller property for $P^t$ 
whenever $t\ge t_0$:  The function  $P^tf$ is continuous whenever $f$ 
is measurable and  bounded; cf.~\Lemma{t:Feller} in the Appendix.
 
$$
\left. 
\mbox{\parbox{.85\hsize}{\raggedright
There is a state $x_0\in\state$ such that,
	for any $x\in\state$ and any open set $O$
	containing $x_0$, we have,
	$$
	P^t(x,O)>0,\qquad \hbox{for all $t\geq 0$ sufficiently large.}
	$$
	}}
\quad\right\}
\eqno{\hbox{\bf (A3)}}
$$
\noindent
Under assumptions~(A2) and~(A3), the chain is $\psi$-irreducible
and aperiodic, with $\psi(\varble)\eqdef P^{t_0} (x_0,\varble)$:
For all $x\in\state$ 
and all $A\in\clB$ such that $P^{t_0}(x_0,A)>0$,
we have,
$$P^t(x,A)>0, \;\;\;\;\;\;\mbox{for all $t\geq 0$
sufficiently large;}$$
see 
\cite[Theorem 6.2.1]{MT}.

Drift conditions are conveniently stated
in terms of the {\em generator} for $\bfmX$.  
In this discrete time setting, for measurable functions 
$f\colon\state\to\RL$ the generator is defined as,
$\clD f :=  P f - f$, that is,
\begin{equation*}
\clD f \, (x) := \Expect[f(X(t+1)) - f(X(t))\mid X(t) = x]
		\qquad x\in\state,
\end{equation*}
for any $f$ for which the expectation is defined for all $x$.  
Fleming's \textit{nonlinear generator} 
\cite{fle78a,fen99a,wu01,fenkur06a,konmey05a}  
is defined via,
\begin{equation}
\clH(F)\eqdef  \log(P e^F) - F,
\label{fle78a}
\end{equation}  
for any measurable function $F$ on $\state$ 
such that $Pe^F$ exists. 

We say 
\cite{konmey03a,konmey05a}
that the {\em Lyapunov drift criterion} (DV3)  
{\em holds with respect to the Lyapunov function 
$V:\state\to(0,\infty]$}, if
		there exist a function 
		$W\colon\state\to [1,\infty)$, 
                a compact set $C\subset\state$,
                and constants $\delta>0$, $b<\infty$,
such that,
$$
\clH(V)\leq -\delta W+b\ind_C\, .
\eqno{\hbox{(DV3)}}
$$
In most of the subsequent results, the following
strengthened version of (DV3) is assumed:
$$
\left. 
\mbox{\parbox{.85\hsize}{\raggedright
	Condition (DV3) holds with respect to
	 functions $V,W$ that are continuously differentiable and have
	compact sublevel sets.  
	}}
\quad\right\}
\eqno{\hbox{\bf (A4)}}
$$
Recall that the sublevel 
sets of a function $F\colon\state\to\Re_+$ are defined by,
\begin{equation*}
C_F(r) =\{ x\in\state : F(x) \le r\},\qquad r\ge 0. 
\end{equation*} 

\subsection{Results}
\label{s:main}

It is assumed throughout the remainder of this section that 
assumptions (A1)--(A4) hold.  It follows that the Markov chain is  
$v$-uniformly ergodic, with $v=e^V$, so that \eqref{e:v-uni} holds for 
a unique invariant probability measure $\pi$ \cite[Theorem~1.2]{konmey05a}.    
The first set of new results in this paper establish a similar conclusion in 
the Banach space $\Lv{1}$.   

\subsubsection{Ergodicity in $\Lv{1}$}
\label{s:lv1ergodic}

The induced operator norm for a linear operator $ \haP \colon\Lvx\to\Lvx$ 
is denoted.
\[
\lll \haP  \lll_v \eqdef 
 \sup \Bigl\{ \frac{\| \haP  f\|_{v}}{\|f\|_{v}} 
:  f\in \Lvx,\ \|f\|_{v}\neq 0\Bigr\}.
\]
On writing $\tilP^t=P^t-\One\otimes\pi$ or, equivalently ,
\[
\tilP^t(x,A) = P^t(x,A) - \pi(A),\qquad x\in\state,\ A\in\clB,
\]
 the bound \eqref{e:v-uni} is expressed as:
\[
\lll \tilP^t \lll_v \le b_0 \rho_0^t \,, \quad t\ge 0.
\] 
Similar notation is adopted for linear operators 
$ \haP \colon\Lv{k}\to\Lv{k}$:
   \[
\lll \haP  \lll_{v,k} \eqdef 
 \sup \Bigl\{ \frac{\| \haP  f\|_{v,k}}{\|f\|_{v,k}} 
:  f\in \Lv{k},\ \|f\|_{v,k}\neq 0\Bigr\}.
\]

Our main result here is the ergodicity of $\bfmX$ in $\Lv{1}$.
In fact, it is stated slightly more generally 
for all spaces $\Lve{1}$, defined as above
with respect to the function
$v^\eta = e^{\eta V}$, for any $\eta\in(0,1]$.

\begin{theorem}
\label{t:ergoMain}
Under assumptions {\em (A1)--(A4)}, for all $\eta \in (0,1]$, 
there is $b_0<\infty$, $t_1 < \infty$ and $\varrho_0<1$ such that:
\begin{equation}
 \lll \tilP^t   \lll_{v^\eta,1} \le b_0\varrho_0^t,\quad t\ge t_1.
\label{e:ergoMain}
\end{equation}
Consequently,  for each $f\in\Lve{1}$ and $t\ge t_1$,
\begin{equation}
\begin{aligned}
\big| \Expect_x[ f(X(t))]  - \pi(f) \big|   &  \leq b_0 \rho_0^t  \|f\|_{v^\eta,1} v^\eta(x),
\\[.2cm]
\text{\it and}\quad 
\Big|  \frac{\partial}{\partial x_i} \Expect_x[  f(X(t))] \Big|   &  \leq b_0 \rho_0^t  \|f\|_{v^\eta,1} v^\eta(x)  \,, \quad 1\le i\le \ell\, .
\end{aligned}
\label{e:v-uni-der}
\end{equation}
\end{theorem}

The proof of the theorem, given in \Section{s:spectral},
is similar to the proof of $v$-uniform ergodicity in prior 
work \cite{konmey03a}. Under assumptions~(A1)--(A4) it is 
shown that the semigroup generated by $\tilP$ has a 
discrete spectrum in $\Lve{1}$,  with spectral radius 
strictly bounded by unity.   

In several of our subsequent results we will need to restrict 
attention to the spaces $\Lve{1}$ for $\eta$ strictly less than 1.
This is justified by the following proposition, 
stated here without proof; it is
a simple consequence of the convexity of the operator $\clH$.

\begin{proposition}
	\label{t:DV3eta}
	If the bound in condition {\em (DV3)}
	 holds, then the same bound holds for any scaling by $\eta\in (0,1)$:
\[
	\clH(\eta V)\leq -\delta \eta W+b\eta \ind_C\, .
\]
\end{proposition}

\subsubsection{Poisson's equation}

For $c\in\Lvx$ the sum \eqref{e:fish-sum} converges in $\Lvx$, 
and $h$ is a solution to the Poisson's equation~(\ref{eq:Poisson}):
$h-Ph=c-\barc$; cf~\cite{glymey96a}. Under appropriate conditions,
we show here that the gradient of $h$ also exists.


Formal term-by-term differentiation of the
definition of $h$ in (\ref{e:fish-sum})
yields:
\begin{equation}
\nabla h =  \sum_{t=0}^{\infty}  \nabla  P^t c.
\label{e:gradFishNoQ}
\end{equation}
This will in fact follow from \eqref{e:ergoMain}, 
once we could establish that $ \lll \tilP^t   \lll_{v,1}$ 
is finite for $t\geq t_1$.   
The proof of \Theorem{t:fish} is given in \Section{s:proofs}.
Recall the definition of the semigroup $\{Q^t\}$ in~(\ref{e:Qt}).
 
\begin{theorem}
\label{t:fish}
Suppose that $c\in\Lvex$, with $\eta\in (0,1]$.  Then, the function $h$ 
in \eqref{e:fish-sum} exists as an element of $\Lvex$. It is a solution 
to Poisson's equation~{\em (\ref{eq:Poisson})}, and it is unique among 
all functions in $\Lvex$ with $\pi$-mean equal to zero.

If $\eta<1$ then we obtain the following additional conclusions:
\begin{romannum} 

\item If $c\in\Lve{0}$   then $h\in\Lve{0}$.

\item If $c\in\Lve{1}$ then $h\in\Lve{1}$,  with gradient 
given in \eqref{e:gradFishNoQ}, and also,
\begin{equation}
\nabla h =\Omega\nabla c := \sum_{t=0}^\infty Q^t \nabla c.
\label{e:gradFish}
\end{equation}
\end{romannum}
\end{theorem}

Note that, in the theorem, the boundedness of $\Omega$ is only established
on the space of functions of the form $\nabla f$ for some $f \in \Lv{1}$.
The first conclusion of the theorem
(that $h$ exists and uniquely solves Poisson's equation)
has been established in \cite{glymey96a,MT};  the remaining conclusions 
are new.  The proof of (ii) is based on a representation of 
the gradient of the semigroup $\{P^t\}$:   
for $f\in\Lve{1}$ with $\eta\in (0,1)$,  and $t\ge 1$:
\begin{equation} 
 \nabla P^t f = Q^t \nabla f \, . 
 \label{e:nablaPQ}
\end{equation} 
This 
and related results are given in \Theorem{t:Pt1vSeparable}.

Theorem~\ref{t:discfish}, given next, states that 
exactly analogous results to those established in \Theorem{t:fish}
for the solution $h$ to Poisson's equation, can also be established for 
the function $h_\alpha$, for any $\alpha\in (0,1)$. Its proof 
is essentially the same as that of \Theorem{t:fish}, and thus omitted.

\begin{theorem}
\label{t:discfish}
Suppose that $c\in\Lvex$, with $\eta\in (0,1]$.  Then, for each 
$0 < \alpha < 1$, the function $h_\alpha$ in \eqref{e:DCOE_disc} exists as 
an element of $\Lvex$.   
It is a solution to the following fixed point equation,
\begin{equation}
c + \alpha P h_\alpha - h_\alpha  = 0
\label{e:DPEquDisc}
\end{equation}
and it is unique among all functions in $\Lvex$.

If $\eta<1$ then we obtain the following additional conclusions:
\begin{romannum} 
\item If $c\in\Lve{0}$   then $h_\alpha \in\Lve{0}$.
\item If $c\in\Lve{1}$ then $h_\alpha \in\Lve{1}$,  with gradient given in \eqref{e:OmDream}:
\begin{equation}
\nabla  h_\alpha = \Omega_\alpha \nabla  c 
:= \sum_{t=0}^{\infty} \alpha^t Q^t  \nabla  c  
\end{equation}
\end{romannum}
\end{theorem}

Once again, the fact that $h_\alpha$ is bounded
and uniquely solves \eqref{e:DPEquDisc},
follows from earlier work 
\cite{CTCN}. The proofs of parts~(i) and~(ii) 
are essentially identical to the proofs of the 
corresponding results in \Theorem{t:fish}.

%
%
%
%
%
%
 %
%

\section{Spectral Theory}
\label{s:spectral}

\begin{proposition}
	\label{t:Banach}
	For any function $v\colon\state\to[1,\infty)$, 
	the normed spaces $\Lvx$, $\Lv{0}$ and $\Lv{1}$ 
	are each Banach spaces.
\end{proposition}
The proof  of \Proposition{t:Banach} is contained in \Section{s:proofs}.
The following subsection concerns spectral theory for an operator 
acting on one of these spaces. 

\subsection{Separability}
\label{s:sep}

A linear operator $T\colon \Lv{k}\to \Lv{k}$ has \textit{finite rank} 
if  there are functions $\{ s_i\} \subset\Lv{k}$, measures 
$\{\nu_j\}\subset \clM_1^v$, and constants $\{m_{ij}\}$ 
such that, 
\begin{equation}
T = \sum_{i,j=1}^N m_{ij} s_i\otimes\nu_j
\label{e:T}
\end{equation} 
where $[s\otimes\nu](x,dy) \eqdef s(x)\nu(dy)$, and $N < \infty$.
We say that a linear operator $\haP\colon\Lv{k}\to\Lv{k}$ 
is \textit{separable in $\Lv{k}$} 
if, for each $\epsy >0$,  there is a finite-rank linear operator $T$ such that  $\lll \haP - T \lll_{v,k} \le \epsy$.    

The \textit{spectrum} $\clS(\haP )\subset \Co$ of a linear operator $\haP \colon\Lv{k}\to\Lv{k}$
is the set of $z\in\Co$ such that the inverse $[Iz - \haP ]^{-1}$ 
does not exist as a bounded linear operator on $ \Lv{k} $.
\archive{Do we show that $\clS_{v^\eta}(P)=\clS_{v^\eta,1}(P)$ for $\eta<1$ ?
\\
Note in September 2017:  I don't know if we care.} 
The {\em spectral radius}
of the semigroup $\{\haP ^n\}$ is denoted
\begin{equation}
\xi_v (\haP ) \eqdef  \lim_{n\to\infty} \lll \haP ^n \lll_v^{1/n}.
\label{e:spectralRad}
\end{equation}
An element 
$z_0\in \clS(\haP )$ is called a
\textit{pole of (finite) multiplicity $n$} if,
for some $\epsy_{1}>0$,
\begin{romannum}  
	\item
	$z_0$ is isolated:   $\{ z\in \clS(\haP ) : |z-z_{0}| \le \epsy_{1} \} =\{z_{0}\}$;
	\item
The associated projection operator $\clP$ has finite rank, where:
\begin{equation}
	\label{e:projection}
	\clP\eqdef \frac{1}{2\pi i}\int_{\partial\{
		z:|z-z_0|\leq\epsy_1
		\}}[Iz-\haP ]^{-1}dz\,.
\end{equation}
%
\end{romannum}
For background see the  decomposition theorem in   
\cite[Theorem~4.4, p.~421]{rienag55}.

\archive{Nov 15:  I fear that Riesz and Nagy would not be happy with us.  We really should take $\nu_j$ in the dual space.  $ \clM_1^v$ is  the dual space for $\Lvx$ and not $\Lv{1}$.  
 Also, isn't it silly to introduce the $m_{ij}$?  After all, 
$\mu_i= \sum_{j=1}^N m_{ij} \nu_j $ lies in  $ \clM_1^v$ for each $i$.  }

The linear operator $\haP :\Lv{k}\to\Lv{k}$
has a {\em discrete spectrum in $\Lv{k}$} if,
for any compact set $C\subset\Co\setminus\{0\}$,
its spectrum $\clS$ has the property that
$\clS\cap C$ is finite and contains
only poles of finite multiplicity.

%
%

The above definition of separability is an extension of separability 
in $\Lvx$ for a  linear operator $\haP\colon \Lvx\to \Lvx$
as defined in \cite{konmey05a}, 
which requires that we can find, for each $\epsy>0$ a positive kernel 
of the form  \eqref{e:T} in which $\{ s_i\} \subset\Lvx$ and  
$\lll \haP  - T \lll_{v,k} \le \epsy$.  Besides the consideration of the 
Banach spaces $\Lv{k}$,  the definition here differs with  \cite{konmey05a} 
in two respects:  First,  positivity of $T$ is not assumed since the 
kernel $\haP $ may not be positive.   Second, in this prior work 
it was assumed that each function $s_i$ and measure $\nu_j $ had support 
on a compact set.   This is not necessary here or in the technical 
results of \cite{konmey05a}.

The following theorem is an extension of Theorem 3.5 of \cite{konmey05a} 
and it provides the fundamental connection between separability and ergodicity.
 

\begin{theorem}[Separability $\Rightarrow$ Discrete spectrum]
	\label{t:sepDiscrete}
	If the linear operator $\haP \colon\Lv{k}\to\Lv{k}$ is bounded and $\haP^{t_1} \colon\Lv{k}\to\Lv{k}$ is separable in $\Lv{k}$ for some $t_1 \ge 1$ and $k\ge 0$, then $\haP $ has a discrete spectrum in $\Lv{k}$.
\end{theorem}

As in the prior work \cite{konmey05a}, separability of the $t$-step transition 
kernel  is established in two steps:  First, it is shown that it can be 
approximated by its truncation to a compact set,
and then the truncated kernel is shown to be separable.

A smooth truncation is the first step in the present paper:
A transition density is approximated using Bernstein polynomials to 
establish that the truncation is separable in $\Lv{1}$.
To simplify notation consider first a  $C^1$ function 
$\varphi\colon[0,1]^N\to\Re$,  with   $N\ge 2$.   
For an integer $m\ge 2$, the Bernstein approximation is given by:
\[
\varphi_m(z) = \sum_{j_1,\dots,j_n=0}^m \varphi\left(\frac{j_1}{m}, \dots, \frac{j_m}{m}\right)
\prod_{i=1}^n \binom{m}{j_i} z^{j_i} (1-z_i)^{m-j_i},\quad z\in[0,1]^N.
\]
The proof of the following can be found in \cite{kin51} for the special 
case $N=2$; also see \cite{but53,haupot77} for related results,
and \cite{majer:12} for a more recent discussion of the general case.


\begin{lemma}
\label{t:bernie} 
The Bernstein polynomials provide the following uniform approximation for any $C^1$ function $\varphi\colon[0,1]^N\to\Re$:
\[
\lim_{m\to\infty} \sup_z \|\varphi(z) - \varphi_m(z) \|_2 = 
\lim_{m\to\infty} \sup_{z,i} \left \|
\frac{\partial}{\partial z_i} \varphi(z) 
-\frac{\partial}{\partial z_i}  \varphi_m(z)\right \|_2 = 0\, .
\]
\end{lemma}


The truncated transition kernel will play the role of $\haP$ in the 
following lemma; its proof is given in Section~\ref{s:bernie}.

\begin{lemma}
\label{t:haPv1Separable}
Suppose $\haP$ has a density $r$ with respect to probability measure $\mu$: For each $x \in \state$ and $A \in \clB$, 
$\haP(x,A) = \int_{A} r(x,y)\mu(dy)$. Suppose moreover that the density $r: \Re^\ell \times \Re^\ell \to \Re^+$ is $C^1$ with compact support. Then, $\haP$ is separable in $\Lv{1}$.  
\end{lemma}

We thus have a roadmap to prove the main results. 
First we consider the case of Markov chains that may not have 
the representation \eqref{e:SP_disc}. 

\subsection{Separability implies ergodicity for general chains}
\label{s:wo}

In this subsection only we consider a general Markov chain evolving 
on $\state=\Re^\ell$, not necessarily of the 
form \eqref{e:SP_disc}. The goal is to generalize the results 
of \Section{s:MainResults}, and also provide an overview of the proofs 
of the main results surveyed there.
 
\Theorem{t:MainGeneral} states that separability in $\Lv{1}$ implies 
ergodicity in this weighted Sobolev space.   Sufficient conditions for 
separability are provided after the proof of the theorem.

\begin{theorem}
\label{t:MainGeneral}
Suppose that the Markov chain $\bfmX$ with transition kernel $P$ satisfies the following conditions, for a continuous function $v\colon\state\to[1,\infty)$: It is  $v$-uniformly ergodic, so that \eqref{e:v-uni}
 holds for each $f\in\Lvx$.  And, for some $t_1 \ge 1$,   $\lll P^t\lll_{v,1}<\infty$ for $t\ge t_1$, and $P^{t_1}$ is separable in $\Lv{1}$.
  
 Then the following conclusions hold:
\begin{romannum}
\item 
The Markov chain is ``ergodic in $\Lv{1}$'': 
		There is $b_0<\infty$ and $\varrho_0<1$ such that
\[
		\lll \tilP^t   \lll_{v,1} \le b_0\varrho_0^t,\quad t\ge t_1.
\]
\item  If, in addition, $\lll P^t\lll_{v,1}<\infty$ for $t \geq 1$, then, for any function $c\in\Lv{1}$, there is a solution to Poisson's equation $h\in\Lv{1}$, with gradient given in \eqref{e:gradFishNoQ}:
\[
\nabla h =  \sum_{t=0}^{\infty}  \nabla  P^t c.
\]
\end{romannum}
\end{theorem}  

Part (ii) of the theorem is based on the following:
\begin{lemma}
\label{t:nablaPQx1} 
For $\eta\in (0,1]$ suppose that $\{g_n\}\subset \Lve{1}$ satisfy $\sup_n \| g_n \|_{v^\eta,1}<\infty$ and
the following limits hold pointwise for continuous functions $g$ and $\zeta$:
\[
\lim_{n\to\infty} g_n(x) = g(x),\quad \lim_{n\to\infty} \nabla g_n(x) = \zeta(x),\quad x\in\state. 
\]
Then, $\nabla g = \zeta$ and $g\in \Lve{1}$.
\end{lemma}

\begin{proof}
For each $n$, $i$, $\alpha$ and $x$ we  have,
\[
 g_n (x+\alpha e^i) -  g_n(x) = \int_0^\alpha   
\partial_i g_n \, (x+t e^i) \, dt,
\]
where, as before, $\partial_i$ denotes the partial derivative 
with respect to the $i$th coordinate.
Letting $n\to \infty $ gives,
\[
 g(x+\alpha e^i) -  g(x) = \int_0^\alpha   \zeta_i(x+t e^i) \, dt.
\]
Continuity of $\zeta$ implies that $\nabla g = \zeta$.  
\end{proof}

\begin{proof}[Proof of \Theorem{t:MainGeneral}] 
It follows from \Theorem{t:sepDiscrete} that  $P$  has a discrete spectrum 
in $\Lv{1}$, and hence this is also true for $P^{t_1}$.
Furthermore, it is straightforward to see that  
the spectrum of $\tilP^{t_1}$ in $\Lv{1}$ is a subset of 
its spectrum in $\Lvx$: $S_{v,1} (\tilP^{t_1}) \subseteq S_v (\tilP^{t_1})$.
Denote the respective spectral radii by 
$\xi_{v,1}(\tilP^{t_1}) $ and   $\xi_{v}(\tilP^{t_1})  $  (recall the definition \eqref{e:spectralRad}).  We obviously have $\xi_{v,1} (\tilP^{t_1})   \le \xi_{v} (\tilP^{t_1})  $.  Also, under $v$-uniform ergodicity we have $\xi_{v}(\tilP^{t_1})  <1$ \cite[Theorem~2.4]{konmey05a}.

The conclusion  $\xi_{v,1} (\tilP^{t_1}) <1$ immediately 
gives (i) for the $t_1$-skeleton chain:   
		There is $b_1<\infty$ and $\varrho_1<1$ such that,
\[
	\lll \tilP^{t_1 k}   \lll_{v,1} \le b_1\varrho_1^k,\quad k\ge 0.
\]
Under the assumption that $\lll P^t\lll_{v,1}<\infty$ for $t\ge t_1$ 
we obtain,
\[
\lll \tilP^{t_1 (k+1) +i}   \lll_{v,1} 
\le 
\bigl(
\max_{0\le j< t_1}  \lll \tilP^{t_1 + j}   \lll_{v,1} 
\bigr) b_1\varrho_1^k\, ,  \quad \text{for each $k\ge 0$ and $0\le i<t_1$,}
\]
which implies (i).
 
Write $\tilc:=c-\barc$.
The ergodicity result (i) is equivalent to the following bound 
for each $c\in\Lv{1}$:  
\[
\max\Bigl\{
\big|P^t \tilc\, (x)   \big|
,
\big| \partial_1  P^t c\, (x)   
\big|,\dots,\big| \partial_\ell  P^t c\, (x)   \big| \Bigr\}
\leq 
b_0 \rho_0^t  \|c\|_{v,1} v(x)  \,, \quad t \geq t_1.
\]
On defining, for each $n\ge 1$,
\[
h_n=   
\sum_{t=0}^n P^t  \tilc\,,
\]
it follows that $h_n\to h$ in $\Lv{1}$  at the same rate, under the assumption  
$\lll P^t\lll_{v,1}<\infty$ for $t \geq 1$.  
And applying \Lemma{t:nablaPQx1},
\[
\nabla h =
\lim_{n\to\infty} \nabla h_n =  
\lim_{n\to\infty}
\sum_{t=0}^n \nabla  P^t  \tilc 
=
\sum_{t=0}^\infty \nabla  P^t  \tilc,
\]
which completes the proof.  
\end{proof}


The next set of results provide conditions under which the assumptions of \Theorem{t:MainGeneral} hold.  It is convenient to strengthen (A2) to $t_0=1$:
$$
\left. 
\mbox{\parbox{.85\hsize}{\raggedright 
There is a continuously differentiable function $p$ on $ \state\times \state$ such that,
\[
		P(x,A) = \int_A p(x,y)\, dy\,,\qquad x\in\state,\ A\in\clB.
\] 
	}}
	\quad\right\}
\eqno{\hbox{\bf (A2')}}
$$ 
  
Assumption~(A3) is maintained, which together with~(A2')
again implies that $\bfmX$ is $\psi$-irreducible
and aperiodic.
%
%
%

The final assumption invokes (DV3) and a similar condition for $\nabla P$.  
The partial derivatives of the density are denoted:
\[
p_i'(x,y) \eqdef \frac{\partial}{\partial x_i} p(x,y)\,, \qquad 1\le i\le \ell.
\]  
$$
\left. 
\mbox{\parbox{.85\hsize}{\raggedright
	\begin{romannum}
\item The transition kernel $P$	satisfies (DV3) with respect to continuous
		functions $V,W$,   a compact set $C\subset\state$,
                and constants $\delta>0$, $b<\infty$.
\item For each $x\in\state,\  1\le i \le\ell$,  		
\[
		\log \int  |p_i'(x,y) | \exp(V(y) - V(x) )\, dy  \le  -\delta W(x) +b\ind_C(x).
	\]
	\end{romannum}		
	}} 
	\quad\right\}
	\eqno{\hbox{\bf (A4')}}
$$  


The drift condition (DV3) is used here and in \cite{konmey05a,konmey09a} to truncate the transition kernel onto a compact subset of the state space.
Denote for $n\ge 1$,
\[
R_n = \{ x \in\Re^\ell :  |x_i|\le n,\ \ 1\le i\le \ell \}.
\]
The function $\Chi_n$ will denote a smooth approximation of the indicator function on this set.  This is based on a function  $\Chi_n^1 : \Re \to [0,1]$   satisfying  $\Chi_n^1(r) = 1$ for $|r| \le n$ and  $\Chi_n^1(r) = 0$ for $|r| \geq n+1 $.  It is assumed that $\Chi_n^1$ is also $C^1$, with:
\[
	\big | \frac{d}{dr} \Chi_n^1 (r) \big |  \leq 2,
	\quad \text{for all } r \in \Re.
\]
The choice is not unique, but fixed throughout the paper.  	In $\ell$ dimensions, define: 
\[
\Chi_n (x) \eqdef \prod_{i=1}^\ell \Chi_n^1(x_i), \quad x \in \Re^\ell.
\]
This function is also $C^1$, equal to $1$ on $R_n$, $0$ on $R_{n+1}^c$, and
\[
	\big | \frac{\partial}{\partial x_i} \Chi_n (x) \big | \leq 2, \quad 1 \leq i \le \ell, \,\, x \in \Re^\ell.
\]

The proof of \Lemma{t:P2truncation} can be found in the Appendix.
\begin{lemma}
\label{t:P2truncation}
Suppose that assumptions~{\em (A2')} and~{\em (A4')} hold.    Then:
\begin{romannum}
\item 
$P^2$ can be approximated by its truncation in $\Lv{1}$:
\[
\lim\limits_{n \to \infty} \lim\limits_{m \to \infty} \lll P^{2} - I_{\Chi_n}P^{2}I_{\Chi_m} \lll_{v,1} = 0
\]

\item
For each $n$, the kernel $ I_{\Chi_n}P^{2}I_{\Chi_n} $ is separable in $\Lv{1}$.
\end{romannum}
\end{lemma}

The assumptions of \Theorem{t:MainGeneral} hold with $t_1 = 2$: 

\begin{theorem} 
\label{t:AssumpMainGen}
Suppose a Markov chain with transition kernel $P$ satisfies  
assumptions {\em (A2')}, {\em (A3)} and~{\em (A4')}. Then:
\begin{romannum}
	\item $P: \Lv{1} \to \Lv{1}$;
	\item $P^2$ is separable in $\Lv{1}$.
\end{romannum}
\end{theorem}

\begin{proof}  
The fact that
$P: \Lv{1} \to \Lv{1}$ is a bounded linear operator follows from 
assumption (A4'), and 	
\Lemma{t:P2truncation} implies that $P^2$ is separable in $\Lv{1}$.
\end{proof}

 
\subsection{Proofs}
\label{s:proofs}

We now return to the   Markov chain described by \eqref{e:SP_disc}.  The 
proposition that follows provides much of the ammunition required to obtain 
a version of \Theorem{t:AssumpMainGen} for this model;  
see \Theorem{t:Pt1vSeparable} below.

The next result concerns separability in 
$\Lvx$: \Proposition{t:PfandQvseparable}~(i) follows from 
Lemma~B.5 of \cite{konmey05a}, and the proof of part (ii) is similar.  
Recall the definition~(\ref{e:Qt}) of the semigroup $\{Q^t\}$, 
which maps $\Re^\ell$-valued functions to$\Re^\ell$-valued functions;
let $Q_{i,j}^t$ denote the $(i,j)$-th component of $Q^t$,
for $1\leq i,j\leq \ell$, $t\geq 0$.


\begin{proposition}
\label{t:PfandQvseparable}
Suppose assumptions {\em (A1)--(A4)} hold. Then, for all $t \ge t_1$:
\begin{romannum}
\item $P^{t}$ is separable in $\Lvx$.
\item $Q_{i,j}^{t}$ is separable in $\Lvx$ for all $1 \leq i,j \le \ell$.  
\end{romannum}
\end{proposition} 
 
Proposition~\ref{t:PfandQvseparable}
is extended in this paper to the weighted Sobolev Banach spaces 
$\Lve{0}$ and
$\Lve{1}$.
The proof  of \Theorem{t:Pt1vSeparable} is contained in the Appendix.  
 
\begin{theorem}
\label{t:Pt1vSeparable}
If assumptions {\em (A1)--(A4)} hold, then:
\begin{romannum}
\item  For all $t\ge t_1$ and $\eta \in (0,1]:$
\begin{alphanum}
	\item  $P^t \colon\Lve{k}\to \Lve{k}$ for $k=0,1$.  
	\item  $Q^t_{i,j}\colon\Lve{0}\to \Lve{0}$ for all
	$1 \leq i,j \leq \ell$. 
	\item
	$ \nabla P^t f = Q^t \nabla f $, if $f\in\Lve{1}$.
	\item $P^t$ is  separable in $\Lve{k}$ for $k=0$ and $k=1$.
\end{alphanum}
\item  Results {\em (a)-(c)} hold for all $t \ge 1$, if $\eta\in (0,1)$.
\end{romannum}
\end{theorem}

\archive{Note to IK and AD in Sept 2017.  I no longer know if this old note is relevant:
\\
This seems too vague to me.  Let's think it over
\\
``Note that each of the results that hold for the function $v$ also hold for $v^\eta$ by applying \Proposition{t:DV3eta}''
}




\archive{Oct 1:  proof too soft.  Let's discuss}
\archive{Oct 5: Proof changed; Need to generalize for $\eta<1$? Because we are calling this theorem in the proof of the next..}
\archive{Nov 15: 	\Proposition{t:DV3eta}
gives us this generalization}


\begin{proof}[Proof of \Theorem{t:ergoMain}]
\Theorem{t:Pt1vSeparable}~(i) states that under assumptions (A1)--(A4), 
$P^{t}$ is separable in $\Lve{1}$, for all $t\ge t_1$ and $\eta \in (0,1]$. It also states that $P^{t}:  \Lve{1} \to \Lve{1}$.
\Theorem{t:sepDiscrete} then implies that $P^{t}$ has a discrete spectrum in $\Lve{1}$.   \Theorem{t:MainGeneral} implies the desired conclusion: 
\[
	\lll \tilP^{t}   \lll_{v^\eta,1} \le b_0\varrho_0^t,\quad t\ge t_1.
\]
\end{proof}
%

\begin{proof}[Proof of \Theorem{t:fish}]
As before, let $\tilc=c-\barc$.
It is obvious that $h$ in \eqref{e:fish-sum} is a solution to Poisson equation, and that its mean is zero.
To establish uniqueness, suppose that $h\in\Lvx$ is any solution with mean zero. We iterate Poisson's equation to obtain,
\[
P^n h = h -\sum_{t=0}^{n-1} P^t \tilc
\]
Since $h\in\Lvx$ with mean zero, we have $\| P^n h\|_v \to 0$  as $n\to\infty$, which establishes that $h$ is equal to the infinite sum in \eqref{e:fish-sum}.  This establishes the first assertions of the theorem.
 
To prove (i), we fix $\eta\in (0,1)$ and $c\in\Lve{0}$.   We have as before that 
 $\| P^t \tilc\|_{v^\eta} \to 0$ as $t\to\infty$ and consequently $h\in\Lvex$. 
 It remains to show that $h$ is continuous.  

Recall from  \Theorem{t:Pt1vSeparable} that $P^t\colon \Lve{0}\to \Lve{0}$ for each $t$.   Since $v$ is assumed to have compact sublevel sets, it follows that $\{P^t \tilc : t\ge 0\}$ are continuous functions that  converge to zero uniformly geometrically fast on compact subsets of $\state$.  This establishes continuity of $h$.

The proof of (ii) requires conclusion~(ii)(c) of \Theorem{t:Pt1vSeparable}:  For $c\in\Lve{1}$, $\eta \in (0,1)$ and $t\ge 1$,
\[
  \bigl( Q^t \nabla c  \bigr)_i    =   \partial_i   P^t c .
\]
\Theorem{t:ergoMain} implies a geometric bound on the right-hand side: For $t \geq t_1$,
\[
 \Big| \partial_i   P^t c\, (x)   \Big|
=
 \Big| \partial_i  \Expect_x[ c(X(t))] \Big|    \leq b_0 \rho_0^t  \|c\|_{v^\eta,1} v^\eta(x)  \,, \quad 1\le i\le \ell\, .
\]

 
Define for each $n\ge 1$,
\[
h_n=   
\sum_{t=0}^n P^t  \tilc\,.
\]
Since $P^t: \Lve{1} \to \Lve{1}$ for each $t \geq 1$, for any finite $n$, we have $h_n \in \Lve{1}$. 
Moreover, since $P^t \tilc \to 0$ in $\Lve{1}$  as $t\to\infty$ at a geometric rate, 
it follows that as $n\to\infty$, $h_n\to h$ in $\Lve{1}$  at the same rate.   

In particular,    
\[
\nabla h =
\lim_{n\to\infty} \nabla h_n =  
\lim_{n\to\infty}
\sum_{t=0}^n \nabla  P^t  \tilc 
=
\sum_{t=0}^\infty Q^t \nabla c
=
\Omega \nabla c,
\]
as claimed.
\end{proof} 

\newpage

\appendix


\section{Appendix}


\subsection{Operator bounds}

We begin with sufficient conditions for the identity \eqref{e:nablaPQ}.  
We first require the following corollary to \Lemma{t:nablaPQx1}:

\begin{lemma}
\label{t:nablaPQlemma}
Suppose that for  some $t \ge 1$ and $\eta\in (0,1]$,
\begin{equation}
P^t \colon\Lve{0}\to \Lve{0}, \qquad Q_{i,j}^t\colon\Lve{0}\to \Lve{0}, \quad  1 \leq i,j \le \ell .
\label{e:PQv0}
\end{equation}  
Then,
$P^t \colon\Lve{1}\to \Lve{1}$  and 
\eqref{e:nablaPQ} holds on $\Lve{1}$:
\[
 \nabla P^t f = Q^t \nabla f  \,, \quad f \in \Lve{1}.
\]
\end{lemma} 

\begin{proof}
For any function $f \in \Lve{1}$ and $n\ge 1$, let $f_n = \Chi_n f$. 
The function $f_n$ and its partial derivatives are continuous, and 
$\sup_n \| f_n \|_{v^\eta,1}<\infty$.  We have $ \lim_{n\to\infty}  \nabla f_n =  \nabla f$, where the limit is continuous by assumption.

We apply  \Lemma{t:nablaPQx1} with $g_n=  P^t f_n$.  To verify the conditions 
of the lemma, first observe that  $\{g_n\}$ converges to $g=P^tf$ by 
dominated convergence.  The limiting function $g$ is continuous 
by~\eqref{e:PQv0}.
From \eqref{e:Qt} it follows that $\nabla g_n = Q^t \nabla f_n$
and, since each $Q^t_{i,j}$ is a bounded linear operator, it follows 
that $\sup_n \| g_n \|_{v^\eta,1}<\infty$.  The final requirement of 
the lemma is convergence of the gradients.  This follows from a second 
application of dominated convergence:
\[
\zeta(x)\eqdef
	 \lim_{n\to\infty}  \nabla g_n (x) = \lim_{n\to\infty} \int Q^t (x,dy) \nabla f_n(y)  = Q^t  \nabla f\, (x).
\]
\Lemma{t:nablaPQx1} then implies the desired conclusion, that 
$\nabla P^t f  = Q^t \nabla f$.   This identity combined with \eqref{e:PQv0} 
then implies that  $P^t \colon\Lve{1}\to \Lve{1}$.
	\end{proof}

A second application of \Lemma{t:nablaPQx1} is in 
the proof of \Proposition{t:Banach}.

\begin{proof}[Proof of \Proposition{t:Banach}]
The proof only requires that each of these function spaces is complete.  
This is an elementary exercise in the case of  $\Lvx$,  and thence $\Lv{0}$.
Completeness of $\Lv{1}$ is established here.

Suppose that $\{f_n\}\subset\Lv{1}$ is a Cauchy sequence.   
Since $\Lv{0}$ is a Banach space,  it immediately follows that there are functions $\{f,\zeta_1,\dots,\zeta_\ell\}\subset \Lv{0}$ such that
\[
\lim_{n\to\infty}  \| f_n - f\|_v = \lim_{n\to\infty}  \| \partial_i  f_n - \zeta_i\|_v = 0,  \quad1\le i\le \ell. 
\]
Consequently, the assumptions of \Lemma{t:nablaPQx1} hold, with $\zeta= \nabla f$ continuous.  Moreover, these limits imply that  convergence of $\{f_n\}$ to $f$ holds in $\Lv{1}$, as required for completeness:  $\lim_{n\to\infty}  \| f_n - f\|_{v,1}=0$.
\end{proof}

\subsection{Truncations}

Several truncation bounds are obtained here. 
The following notation will be useful: For any operator $Z$ on 
$\Lvx,\Lv{0}$ or $\Lv{1}$,
we write $Z_n \tov Z$, if
\[
\lim\limits_{n \to \infty} \lll Z_n - Z  \lll_v = 0.
\]
The elementary observation stated below without proof,
is used to avoid establishing one-sided truncation bounds;
recall the definition of the functions $\{\Chi_n\}$ in Section~\ref{s:wo}.
\begin{lemma}
\label{t:TruncLR}
Suppose that $Z$ is a bounded linear operator on a Banach space 
of functions on $\state$, with induced operator norm $\lll\varble \lll$.   
If $Z$ it can can be approximated by its truncation on both sides,
\archive{ it's = it is !
Hint: The word involved is small and it's contained in this sentence. Yet the two rules are actually quite easy to remember. Rule 1: When you mean it is or it has, use an apostrophe. Rule 2: When you are using its as a possessive, don't use the apostrophe.
Its vs. It's - Grammar \&\ Punctuation | The Blue Book of Grammar and ...}
\[
	\lim\limits_{n \to \infty} \lll Z - I_{\Chi_n}Z I_{\Chi_n} \lll = 0,
\]
then $Z$ can be approximated by its truncation on either side:	
\[
		\lim\limits_{n \to \infty} \lll Z - Z I_{\Chi_n} \lll = 
		\lim\limits_{n \to \infty} \lll Z - I_{\Chi_n} Z \lll = 0.
\] 		
\end{lemma}

\begin{proof}[Proof of \Lemma{t:P2truncation}]
It is only necessary to prove (i), since the implication (i) $\Rightarrow$ (ii) follows from \Lemma{t:haPv1Separable}.

Assumption (A2') along with part (i) of (A4') implies that the transition 
kernel can be approximated by its left truncation: In $\Lvx$ we have
$I_{\Chi_n} P\tov  P$ (see \cite[Lemma B.4]{konmey05a}),  and hence:
\begin{equation}
(I_{\Chi_n} P)^2 \tov P^2.
	\label{e:leftTruncP}
\end{equation} 
Furthermore, assumptions~(A2') and~(A4')  imply a bound of the form,
\[
\big|I_{\Chi_n} P I_{\Chi_n} (x,A) \big| \leq \beta_n^0(A),
\quad 1 \leq i \leq \ell, \ x\in\state, \ A\in\clB,
\]
where $\beta_n^0$ is a positive measure with compact support, and hence  $\beta_n^0(v) < \infty$.  Therefore, for all $x \in \state$ and $A \subset \clB$,
\[
\begin{aligned}
\big|  (I_{\Chi_n} P)^2 (x,A) \big| 
& = \big| \int\limits_{R_{n+1}}  \Bigl\{  \Chi_n (x) P(x,dy)  \Chi_n(y) P(y,A)  \Bigr\} \big| \\
& \leq \int\limits_{R_{n+1}} \beta_n^0(dy) P(y,A)  \eqdef \beta_n(A)\,,
\end{aligned}
\]
with $\beta_n(v)<\infty$ since both $\lll P\lll_v$ 
and $\beta_n^0(v)$ are finite.   
Therefore, for any $f\in\Lvx$,
\[
\begin{aligned}
\| (I_{\Chi_n} P)^2 I_{\Chi_m} f  -  P^2 f  \|_v
&  
\le
\| (I_{\Chi_n} P)^2 ( 1-\Chi_m)f    \|_v
+
\| (I_{\Chi_n} P)^2   f  -  P^2 f  \|_v  
\\[.2em]
&\le    \beta_n(v I_{R_m^c}) \|f\|_v
						+  
\| (I_{\Chi_n} P)^2   f  -  P^2 f  \|_v ,
\end{aligned}
\]
and applying \eqref{e:leftTruncP}, 
\[
\lim\limits_{m,n \to \infty} 
\lll  (I_{\Chi_n}P)^{2}I_{\Chi_m} -P^{2} \lll_{v} = 0.
\]

To complete the proof of (i), it remains to be shown that,
for each $1 \leq i \leq \ell$, 
\begin{equation}
\lim\limits_{m,n \to \infty} 
	 \lll \partial_i ( I_{\Chi_n} P)^2 I_{\Chi_m}
	 		-\partial_i P^2   \lll_{v} = 0\,,
\label{e:P2der-trunc}
\end{equation}
where, again, $\partial_i$ is shorthand for ${\partial}/{\partial x_i}$.
The proof follows exactly the same steps as before:
Assumption (A2') and part (ii) of (A4') imply that $P_i'$ can be truncated on the left: $I_{\Chi_n} P_i'\tov P_i' $ for each $i$;  that is,
\[
\lim\limits_{n \to \infty} \lll \partial_i P - I_{\Chi_n} \partial_iP \lll_{v} = 0\, ,\quad 1\le i\le \ell\, .
\]
From this, and the prior conclusion $I_{\Chi_n} P\tov  P$, we obtain
\begin{equation} 
\lim\limits_{n \to \infty}
	   \|\partial_i (I_{\Chi_n} P)^2  f  
  	-\partial_iP^2  f   \|_v  =0\, .
	\label{e:leftTruncPidash}
\end{equation}	
	
Furthermore, the two assumptions imply a bound of the form
\[
	\Big|\partial_i I_{\Chi_n} P I_{\Chi_n} (x,A) \Big| \leq \gamma_n^0(A),
					 \quad 1 \leq i \leq \ell, \ x\in\state, \ A\in\clB,
\]
 where $\gamma_n^0$ is a positive measure with compact support, and hence  $\gamma_n^0(v) < \infty$.  
 	Therefore, for all $x \in \state$, $A \subset \clB$ and $1 \leq i \leq \ell$,
\[
\begin{aligned}
	\big|  \partial_i (I_{\Chi_n} P)^2 (x,A) \big| 
& = \bigg| \int\limits_{R_{n+1}} \frac{\partial}{\partial x_i} \Bigl\{  \Chi_n (x) P(x,dy)  \Chi_n(y) P(y,A)  \Bigr\} \bigg| \\
	& \leq \int\limits_{R_{n+1}} \gamma_n^0(dy) P(y,A)  \eqdef \gamma_n(A).
\end{aligned}
\]
It follows that for all $ f \in \Lvx$, and $ 1 \leq i \leq \ell$, 
\[ 
\begin{aligned}
  \|\partial_i (I_{\Chi_n} P)^2 I_{\Chi_m} f  
  	- \partial_i P^2 f  \|_v 
	&	\le 
	   \|  \partial_i (I_{\Chi_n} P)^2  (1 - \Chi_m) f \|_v 
	   +
	   	   \|\partial_i (I_{\Chi_n} P)^2  f    	-\partial_iP^2  f   \|_v 
	\\[.2em]
	&\le
	 \gamma_n(v I_{R_m^c}) \|f\|_v 
	 +
	   \|\partial_i (I_{\Chi_n} P)^2  f    	-\partial_iP^2  f   \|_v 
\end{aligned}
\]
Combining this with \eqref{e:leftTruncPidash} implies that \eqref{e:P2der-trunc} holds,   
and this completes the proof of part (i), as required.
\end{proof}


The next results concern the nonlinear state space model.   
\Lemma{t:Feller} follows directly from the assumptions. 
Recall the discussion of the (strong) Feller property in
Section~\ref{s:irred}. As before,  
$Q_{i,j}^t$ denotes the $(i,j)$-th component of $Q^t$,
for $1\leq i,j\leq \ell$, $t\geq 0$.


\begin{lemma}
\label{t:Feller}
Suppose that assumptions {\em (A1)--(A4)}
hold,   and  let $Z_t$ denote any one of the kernels $P^t$ or $Q^t_{i,j}$ with   $t\ge 1$ and $1\le i,j\le \ell$.   
\begin{romannum}
\item   The   {\em Feller property} holds for $Z_t$, for $t \geq 1$ and the {\em strong Feller property} holds for  $Z_t$, when $t \geq t_0$. Moreover, the following stronger properties hold:
\[
\begin{aligned}
Z_t : \Lve{0} \to \Lve{0} , \quad t\ge 1,
\\	
Z_t : \Lvex \to \Lve{0} ,\quad t\ge t_0 .
\end{aligned}
\]
\item
 For each $n\ge1$ and $\eta\in (0,1]$,
\[
\begin{aligned}
Z_t I_{\Chi_n}&: \Lve{0} \to \Lve{0} , \quad t\ge 1,
\\
Z_t I_{\Chi_n}&: \Lvex \to \Lve{0} ,\quad t\ge t_0.
\end{aligned}
\]
\item
Suppose  that for some $\eta\in (0,1]$, $t\ge 1$ and every $g\in\Lvex$,   
\[
\lim_{n\to\infty} Z_t I_{\Chi_n} g = Z_t g, 
\]
where the convergence is uniform on compact subsets of $\state$.  Then,    
\[
\begin{aligned}
Z_t  &: \Lve{0} \to \Lve{0}\,, 
\\
  \text{and}\quad Z_t & : \Lvex \to \Lve{0} ,\quad \text{provided} \,\, t\ge t_0 \,.
\end{aligned}
\]
\end{romannum}
\end{lemma}

The proof of the next result is also elementary.
\begin{lemma}
\label{t:PfandQtruncation} 
Suppose the conclusions of \Proposition{t:PfandQvseparable} are true,
that is,
for each $t \ge t_1$:
\begin{alphanum}
\item $P^{t}$ is separable in $\Lvx$;
\item $Q_{i,j}^{t}$ is separable in $\Lvx$ for any pair $1 \leq i,j \le \ell$.
\end{alphanum}
Then, the kernels $P^{t}$ and $Q^{t}$ can be approximated by their truncations:
\begin{romannum}
\item  $ \lim\limits_{n \to \infty} \lll P^{t} - I_{\Chi_n}P^{t}I_{\Chi_n} \lll_{v,1} = 0 $;
\item  $ \lim\limits_{n \to \infty} \lll Q_{i,j}^{t} - I_{\Chi_n}Q_{i,j}^{t}I_{\Chi_n} \lll_v = 0 $ for any pair $1\leq i,j \le \ell$.
\end{romannum} 
\end{lemma}


\begin{proof}
The fact that the kernels can be approximated in $\Lvx$ by their truncations 
for each $t\ge t_1$ follows directly from the assumption that they are 
separable: We have,
\begin{equation}
\begin{aligned}
I_{\Chi_n}P^{t}I_{\Chi_n} & \tov P^{t}
\\
I_{\Chi_n}Q^{t}_{i,j}I_{\Chi_n} & \tov Q^{t}_{i,j}
, \quad 1 \leq i,j \le \ell.
\end{aligned}
\label{eq:trunclv}
\end{equation}
In particular,  part (ii) is immediate.

To complete the proof of (i), it remains to be shown that there is a vanishing sequence $\{\epsy(n)\}$ such that,
 for any function $f \in \Lv{1}$,
\[
  \| \partial_i \{ P^{t}f\}  - \partial_i  \{I_{\Chi_n}P^{t}I_{\Chi_n}f\} \|_v \le \epsy(n) \|f\|_{v,1}.
\]
\Lemma{t:Feller} along with \eqref{eq:trunclv} implies that the 
assumptions of \Lemma{t:nablaPQlemma} are satisfied (with $t \geq t_1$):
\[
P^t \colon\Lve{0}\to \Lve{0}, \qquad Q_{i,j}^t\colon\Lve{0}\to \Lve{0}, \quad  1 \leq i,j \le \ell.
\]
Now, applying the product rule gives,
\[
\partial_i  \{I_{\Chi_n}P^{t}I_{\Chi_n}f\}
		= \{\partial_i \Chi_n \} \{ P^{t} I_{\Chi_n} f\} 
			+ I_{\Chi_n}  (Q^{t} \nabla  (f \Chi_n) )_i,
\]
with the second term justified applying \Lemma{t:nablaPQlemma}.

The first term can be bounded bounded,
\[
\|  \{\partial_i \Chi_n \} \{ P^{t} I_{\Chi_n} f\}  \|_v 
\le \epsy_1(n) \| f\|_v \le \epsy_1(n) \| f\|_{v,1}\,,
\]
where,  
\[
 \epsy_1(n)  =   \bigl(\max_i  \|\partial_i \Chi_n \|_{\infty} \bigr) \lll I_{R_n^c} P^{t}   \lll_v \le 2 \lll I_{R_n^c} P^{t}   \lll_v.
\] 
The pre-multiplication by $I_{R_n^c} $ is justified since $\nabla  \Chi_n = 0$ on $R_n$. 
Equation \eqref{eq:trunclv}  along with \Lemma{t:TruncLR} implies that 
$\lll I_{R_n^c} P^{t} \lll_v \to 0$ as $n \to \infty$, and hence $\lim_{n\to\infty} \epsy_1(n) =0$.
Therefore, 
\[
\begin{aligned}
  \| 
  \partial_i \{ P^{t}f\}  - \partial_i  \{I_{\Chi_n}P^{t}I_{\Chi_n}f\} \|_v   
  	& \le \epsy_1(n) \|f\|_{v,1} 
  \\
	&  \quad +  \|  (Q^{t} \nabla  f)_i -  I_{\Chi_n}  (Q^{t} \nabla  (f \Chi_n))_i  \|_v .
\end{aligned}
\] 
Once more applying equation \eqref{eq:trunclv}, it is 
straightforward to see that there is a vanishing sequence $\{\epsy_2(n)\}$
such that:
\[
\|  (Q^{t} \nabla  f)_i -  I_{\Chi_n}  (Q^{t} \nabla  (f \Chi_n))_i  \|_v \le \epsy_2(n) \|f\|_{v,1} ,\quad n\ge 1. 
\]
This  completes the proof of the lemma.
\end{proof}

The justification of the representation \eqref{e:gradFish} for $\nabla h$   requires a different set of truncation arguments:

\begin{lemma}
	\label{t:PfandQtruncation-eta}
	Let $\eta\in(0,1)$. For each $t\ge 1$, the kernels $P^{t}$ and 
	$Q^{t}$ can be approximated in $\Lve{1}$ by their truncations 
	on the right:
	\begin{romannum}
		\item  $ P^{t}I_{\Chi_n} \,\,\,\, \tove P^{t}$;
		\item  $ Q_{i,j}^{t}I_{\Chi_n} \tove Q_{i,j}^{t}, \quad$ 
		for any pair $1\leq i,j \le \ell$.
	\end{romannum} 
\end{lemma}  

\begin{proof}
	Let $\eta\in(0,1)$, take $f\in\Lve{1}$, and let
	$f_n \eqdef I_{\Chi_n}f$. Then, for all $t \geq 1$,
\[
\begin{aligned}
\big| P^t f(x) - P^t f_n(x) \big| & \le 
\bigg| \int_{R_n^c} P^t(x,dy) f_n(y) \bigg| 
\\
& \leq \|f\|_{v^\eta} \int_{R_n^c} P^t (x,dy) v(y) 
	\bigg( \frac{v^\eta (y)}{v(y)} \bigg) \\
& \leq \|f\|_{v^\eta} \Big[\sup_{y' \in R_n^c} v^{\eta - 1}(y')\Big]
	\int_{R_n^c} P^t (x,dy) v(y) \\
& \leq \|f\|_{v^\eta} \epsy(n),
\end{aligned}
\]
where $\epsy(n) \to 0$ as $n \to \infty$.   The last step follows from the fact that $\lll P^t\lll_v<\infty$  under (DV3),   
and   $v(x) \to \infty$ as $\|x\|_2 \to \infty$  because
$v$ has compact sublevel sets under assumption~(A4).
	
Under assumption~(A1), and using the same arguments as above, we have:
\[
\lim\limits_{n \to \infty} \| Q_{i,j}^{t} f - Q_{i,j}^{t} f_n \|_{v^\eta} = 0, \,\, \text{for all} \,\, 1 \leq i,j \leq \ell.
\]
\end{proof}

The following strengthening of the Feller property is another step in the proof of \Theorem{t:Pt1vSeparable}.
\begin{proposition}  
\label{t:PQv0lemma}
Under assumptions {\em (A1)--(A4)}:
\begin{romannum}
\item  
For all $t \geq t_1$, and $\eta=1$,
\begin{equation}
P^t \colon\Lve{0}\to \Lve{0}, \qquad Q_{i,j}^t\colon\Lve{0}\to \Lve{0}, \quad  1 \leq i,j \le \ell .
\label{e:Lve0}
\end{equation}

\item   The conclusions \eqref{e:Lve0} hold for all $t \geq 1$  when $\eta \in (0,1)$.
	\qed
\end{romannum}
\end{proposition}

\begin{proof} 
	\Lemma{t:PfandQtruncation} (i) along with \Lemma{t:TruncLR} implies that for any function $g\in\Lv{0}$ and all $t \ge t_1$ we have, 
\[
	\lim\limits_{n \to \infty} P^{t} I_{\Chi_n} g   =  P^{t} g,
\]  
	where the convergence is uniform on compact subsets of $\state$.   It then follows from \Lemma{t:Feller} that $P^{t} \colon\Lve{0}\to \Lve{0}$ for any $\eta\in (0,1]$.
	
Similarly, using \Lemma{t:PfandQtruncation} (ii),
\[
	\lim\limits_{n \to \infty}  Q_{i,j}^{t} I_{\Chi_n} g   =   Q_{i,j}^{t} g,
	\]
	for any $g \in \Lv{0}$. This again implies that $Q_{i,j}^{t} \colon\Lve{0}\to \Lve{0}$, $\eta\in (0,1]$, from \Lemma{t:Feller}. 
	This completes the proof of part (i) of the proposition.

The proof of part (ii) follows exactly in the same manner, using \Lemma{t:PfandQtruncation-eta} (instead of \Lemma{t:PfandQtruncation}) along with \Lemma{t:Feller}. 
\end{proof}
 
 \begin{proof}[Proof of  \Theorem{t:Pt1vSeparable}]
First consider part (i).
\Proposition{t:PQv0lemma} establishes (b), and part of~(a): 
For all $t \geq t_1$, and $\eta \in (0,1]$,
\begin{equation}
P^t \colon\Lve{0}\to \Lve{0}, \qquad \text{and} \qquad Q_{i,j}^t\colon\Lve{0}\to \Lve{0}, \quad  1 \leq i,j \le \ell .
\label{e:t:PQv0lemma}
\end{equation}
Applying \Lemma{t:nablaPQlemma} we obtain the remainder of (a), and also (c).

 	
Assumption~(A2) implies that $P^{t_1}$ has a density which is $C^1$.  Furthermore, from \Lemma{t:PfandQtruncation}, we conclude that under 
assumptions (A1)--(A4), $P^{t_1}$ can be approximated by its truncation $I_{\Chi_n}P^{t_1}I_{\Chi_n}$ in $\Lv{1}$.   \Lemma{t:haPv1Separable} 
therefore completes the proof of (d).
	
Next, consider part (ii). \Proposition{t:PQv0lemma} again establishes (b).
Part (ii) of \Proposition{t:PQv0lemma} states that \eqref{e:t:PQv0lemma} holds for each   $t \geq 1$ and $\eta \in (0,1)$.  Consequently, results (a) and (c) follow as before by applying \Lemma{t:nablaPQlemma}.
\end{proof} 
%
%

\subsection{Separability and Bernstein polynomials}
\label{s:bernie}

\begin{proof}[Proof of \Lemma{t:haPv1Separable}]
Let $r_v (x,y) := r(x,y) v(y)$. For any function $g \in \Lvx$, we have:
\[
\haP g\, (x) = \int r_v (x,y) g(y) v^{-1} (y) \mu(dy).
\]
Since $v$ is assumed to be $C^1$, 
$r_v$ is also $C^1$ with compact support. 

Choose $n\ge 1$ such that $r_v(x,y) = 0$ on $(R_n \times R_n)^c$.
Therefore, for any given $\epsy>0$ there exists a Bernstein's polynomial $r^{\epsy_0}_v$
such that, for all $(x,y) \in R_{n+1} \times R_{n+1}$,
 \[
\begin{aligned}
  \Big | r_v(x,y) - r^{\epsy_0}_v(x,y) \Big | & \leq \epsy,
\\
\text{and} \quad
  \Big | \frac{\partial}{\partial x_i} r_v(x,y) - \frac{\partial}{\partial x_i} r^{\epsy_0}_v(x,y) \Big | & \leq \epsy, \quad 1\leq i \le \ell.
\end{aligned}
\]  
The approximating polynomial can be expressed in the suggestive form,
\[
r^{\epsy_0}_v(x,y)  =  \sum_{i=1}^N   s_i^0(x)  r_i^0(y)
\]
Truncating the approximation smoothly as 
$r^\epsy_v(x,y) = \Chi_{n}(x)\Chi_{n}(y) r^{\epsy_0}_v(x,y)$,
we obtain a function supported on $R_{n+1}\times R_{n+1}$,
\[
r^\epsy_v(x,y)  =  \sum_{i=1}^N   s_i(x)  r_i(y),
\]
with $s_i =  \Chi_{n}s_i^0$ and  $r_i =  \Chi_{n} r_i^0$.
It is then straightforward that,
\[
\begin{aligned}
\sup_{x,y} \Big | r_v(x,y) - r^\epsy_v(x,y) \Big | &\leq \epsy,\\
\text{and} \quad
\sup_{x,y} \Big | \frac{\partial}{\partial x_i} r_v(x,y) - \frac{\partial}{\partial x_i} r^\epsy_v(x,y) \Big | &\leq \epsy, \quad 1\leq i \le \ell,
\end{aligned}
\]
where the suprema are over $(x,y)\in \state \times \state$. 

The following approximating kernel has finite rank:
\[
T_\epsy (x,dy) = r^\epsy_v (x,y)  v^{-1} (y) \mu(dy).
\]
We also have, 
\[
\begin{aligned}
\Big | \haP g(x) - T_\epsy g(x) \Big | 
& \leq  \int \Big | r_v(x,y) - r^\epsy_v(x,y) \Big | \Big | 
	\frac{g(y)}{v(y)} \Big | \mu(dy) \\
& \leq   \sup_{x,y} \Big | r_v(x,y) - r^\epsy_v(x,y) \Big |   
	\sup_z\Big | \frac{g(z)}{v(z)} \Big |   \\
& \leq    \epsy \|g\|_v,
\end{aligned}
\]
and,
\[
\begin{aligned}
\Big | \frac{\partial}{\partial x_i} \haP g(x) - \frac{\partial}{\partial x_i} T_\epsy g (x) \Big | & =  \Big | \frac{\partial}{\partial x_i}  \int \Delta_r^\epsy (x,y) \frac{g(y)}{v(y)} \mu(dy)  \Big | \\
& =  \Big | \lim_{{\delta} \to 0} \frac{1}{\delta} \int \Big ( \Delta_r^\epsy(x + {\delta e^i},y) - \Delta_r^\epsy(x,y) \Big ) \frac{g(y)}{v(y)} \mu(dy)  \Big |,
\end{aligned}
\]
where $\Delta_r^\epsy = r_v - r^\epsy_v$,   
and $e^i$ denotes the $i^{th}$ basis vector in $\Re^\ell$.

Since, both $r_v$ and $r_v^\epsy$ are $C^1$,  the mean value theorem gives,
\[
\frac{1}{\delta}\Big |  \Delta_r^\epsy (x + {\delta e^i},y) - \Delta_r^\epsy(x,y) \Big | = \Big | \frac{\partial}{\partial x_i}  \Delta_r^\epsy(\barx_i,y) \Big | ,
\]
for some $\barx_i \in (x,x+{\delta e^i})$.  The right-hand side is 
uniformly bounded over all $\delta\in (0,1]$ and thus,
by dominated convergence,
\[
\begin{aligned}
\Big | \frac{\partial}{\partial x_i} \haP g(x) - \frac{\partial}{\partial x_i} T_\epsy g(x) \Big | & \leq   \int \limsup_{{\delta} \to 0} \frac{1}{{\delta}} \Big | \Delta_r^\epsy(x + {\delta e^i},y) - \Delta_r^\epsy(x,y) \Big | \Big | \frac{g(y)}{v(y)} \Big | \mu(dy)  \\
& \leq     \sup_{x,y} \Big | \frac{\partial}{\partial x_i} r_v(x,y) - \frac{\partial}{\partial x_i} r^\epsy_v(x,y) \Big |   \|g\|_v 
\\
& \leq   \epsy \|g\|_v.
\end{aligned}
\]
This completes the proof of separability of $\haP$ in $\Lv{1}$. 
\end{proof}

\newpage
 
\newpage

\def\cprime{$'$}\def\cprime{$'$}

\end{document}